\documentclass[a4paper,11pt]{article} 
\usepackage{amsmath}
\usepackage[retainorgcmds]{IEEEtrantools}
\usepackage{graphicx}
\usepackage{amsfonts}
\usepackage{amsthm}
\usepackage{setspace}
\usepackage{mathtools}
\usepackage{amssymb}
\usepackage[mathscr]{euscript}
\usepackage{cite}
\usepackage{pictexwd,dcpic}

\newtheorem{thm}{Theorem}[section]
\newtheorem{dfn}{Definition}[section]
\newtheorem{exm}{Example}[section]
\newtheorem{crl}{Corollary}[section]
\newtheorem{prp}{Proposition}[section]

\newtheorem{cnj}{Conjecture}[section]
\newcommand{\gen}[1]{\left< #1 \right>}

\newcommand{\Gen}{\text{Gen}}

\newcommand{\scr}[1]{\mathscr{#1}}

\newcommand{\ch}{\text{ch}}
\newcommand{\lu}{\text{lu}}
\newcommand{\lk}{\text{link}}
\author{Cole Hugelmeyer} 
\title{A Construction on Operads Applicable to Homology Operations and the Minimal Model} 
\begin{document} 
\maketitle 

\begin{abstract}
We define a construction on operads which yields a new description of the minimal model. The construction also allows us to define algebraic structures on the homology of chain complexes with homologously trivial operad algebra structures, thus exposing nontrivial structure where none is apparent.
\end{abstract}

{\bf Keywords:} operad, minimal model, homotopy algebra

\section{Introduction and Notation}
$\;\;\;\;\;\;\;$In this paper we define a construction on symmetric operads which attaches a graded operad to a differential graded operad by universally letting the generators of the first operad differentiate to elements of the second. We call the resulting operad a universal linking operad, and we offer a combinatorial description of its presentation and homology. This construction has many surprising properties. For instance, given a chain complex with a homologously trivial operad algebra structure, we are able to construct nontrivial algebra structures over the homology of a universal linking operad on the homology of the complex. More surprisingly, the minimal model of any operad can be constructed by an infinite iteration of the universal linking construction, where the free operad on the lowest arity homology of the previous iteration is linked on at each step.  The notion of the minimal model of an operad arises naturally from attempts to define algebraic structures up to homotopy \cite{vogt}\cite{berger}\cite{kapranov} and finds applications in string theory \cite{markl}.  The construction in this paper is therefore highly applicable to these areas because it offers a new perspective on the combinatorial nature of the minimal model.

In this paper all operads, unless otherwise idicated, are assumed to be graded.  By $\gen{G}$ we denote the free operad on a $\Sigma$-module $G$. By $\gen{G|R}$ we mean the operad cokernel of the map $\gen{R}\to \gen{G}$, where it is assumed that $R$ has a map to the underlying $\Sigma$-moodule of $\gen{G}$.  Let $U$ be the forgetful functor from operads to $\Sigma$-modules, and let $\varepsilon$ and $\eta$ be the unit and counit of adjunction between operads and $\Sigma$-modules respectively.   Let $\text{Gen}(\scr{O})$ denote the minimal generating $\Sigma$-module of an operad $\scr{O}$.  Lower case $s$ will denote suspension.  Generally, the structural maps for operads will be shortened to $a(\otimes_{i = 1}^n b_i)$ to mean element $a$ of arity $n$ composed with all $b_i$.  All chain complexes are assumed to be only nontrivial in nonnegative degree.  We will often refer to the tree basis of a free operad.  See \cite{grobner} \cite{kapranov}.  When we do this, assume we have chosen a $\Sigma$-invaritant basis of the generating $\Sigma$-module of the operad, and the basis elements of the operad are composition trees with nodes decorated by these specified basis elements.

\section{Chain Complex Adjunction}

Let $\scr{O}$ be an operad in the category of graded vector spaces over a field $\mathbb{K}$.  Let $\iota$ and $\partial$ denote maps $\scr{O}\to\gen {U(\scr{O}) \oplus s^{-1}U(\scr{O})}$ given by the identification of $\mathscr{O}$ with the first term of the direct sum, and second term of the direct sum respectively.  We define $\ch \scr{O} = \gen {U(\scr{O}) \oplus s^{-1}U(\scr{O})}/R$ where $R$ is the ideal generated by the relations below.   \begin{align}\partial(a(\otimes_{i = 1}^n b_i)) = \partial a(\otimes_{i = 1}^n \iota b_i ) + \sum_{i = 1}^n(-1)^{|a| + \sum_{j = 1}^{i-1}|b_j|} (\iota a(\otimes_{j = 1}^n\varphi_{ij}b_j)) \\
\iota(a(\otimes_{i = 1}^n b_i)) = \iota a(\otimes_{i = 1}^n \iota b_i) \end{align}
where $\varphi_{ij}$ is $\partial$ when $i = j$ and $\iota$ otherwise. 
\begin{prp}
The functor \emph{ch} is a left adjoint to the forgetful functor, $V$, from chain complex operads to graded vector space operads.  
\end{prp}
\begin{proof}
Given a graded vector space operad $\scr{P}$ and a chain copmplex operad $\scr{Q}$, there is a natural map $$a:\hom(\scr{P},V\scr{Q})\to \hom(\gen {U(\scr{P}} \oplus s^{-1}U(\scr{P})),\scr{Q})$$ given by the differential on $\scr{Q}$ and the universal property for the free operad functor.  There is also a natural map $$b: \hom(\ch(\scr{P}), \scr{Q}) \to \hom(\gen {U(\scr{P}) \oplus s^{-1}U(\scr{P})},\scr{Q})$$ given by composition with the canonical projection. Due to the chain homomorphicity conditions for the differential, and the operad homomorphicity for the inclusion, $R$ is mapped trivially in the image of $a$, which means the image $a$ is the same as the image of $b$, and since both $a$ and $b$ are injective, this means that there is a natural bijection $b^*a:\hom(\scr{P},V\scr{Q})\to \hom(\ch \scr{P}, \scr{Q})$ given by lifting $a$ up $b$.  This yields an adjunction.  
\end{proof}

\begin{prp}
Let $\scr{O} = \gen{G|R}$ be generated by generators \emph{$G $} and relations \emph{$R\subseteq \gen{G}$}.  Then \emph{$\ch(\scr{O}) = \gen{G\oplus s^{-1}G| R\oplus s^{-1}R}$}.  
\end{prp}

\begin{proof}
Relations (1) and (2) imply that any element of $\ch(\scr{O})$ can be factored into a composition of elements in $G\oplus s^{-1}G\to \ch{\scr{O}}$.  By differentiating the relations for $G$, we see that $\ch(\scr{O})$ must at least have the relations $R\oplus s^{-1}R$, but from the universal property of the adjunction, we see that these are the only relations.  
\end{proof}

\begin{prp}
For any operad $\scr{O}$, there is an isomorphism  \emph{$$\theta : U\ch(\scr{O})\to U\scr{O}\circ \eta \,s^{-1}\,U\scr{O}$$} with the differential given by $\partial(1\otimes a) = 0$ and when \emph{$\text{arity}(a) = n >1$},  $$\partial(a\otimes (\otimes_{i = 1}^n b_i)) = a(\otimes_{i = 1}^n b_i).$$
\end{prp}

\begin{proof}
Let $\scr{O} = \gen{G|R}$ as in Proposition 2.2, with $G = U\scr{O}$.  We may then use the derivatives of the composition relations in $G$ to express any composition tree in $\gen{G\oplus s^{-1}G}$ uniquely as a sum of trees which have only one element of $G$ at the root, and all other nodes in $s^{-1}G$.  To see this, note that whenever an element of $G$ appears further along the tree than an element of $s^{-1}G$, an application of the derivative of a composition relation will give a signed sum of a tree with the differential switched to the farther node and a tree where the two nodes have been combined into something in $s^{-1}G$.  Repeating this process will yield the desired result.  We have therefore constructed an epimorphism $\phi: \eta(G\oplus s^{-1}G) \to G \circ \eta s^{-1} G = G\circ \eta \, s^{-1}\,U \scr{O}$.  Which extends down to a map $\theta: U \ch(\scr{O})\to U\scr{O}\circ \eta s^{-1}U\scr{O}$.  We now see that the differential is the one we want by applying relation (1) to the trees in our construction.  To determine that $\theta$ is an isomorphism, we observe that, by construction, the kernel of $\phi$ is the ideal generated by the composition relations and their derivatives.  
\end{proof}

\begin{crl}
For all operads $\scr{O}$, we have $H_*\ch\scr{O} = 0$.
\end{crl}

\begin{proof}
The chain complex described in Proposition 2.3 is acyclic.  
\end{proof}

\section{Universal Linking Operads}

Let $\scr{P}$ and $\scr{Q}$ be operads $\gen{G|R}$ and $\gen{G'|R'}$ respectively, where $\scr{Q}$ is equipped with a differential.  Suppose we are given a map $\varphi: s^{-1}G\to U\gen{G'}$. Then there are maps $$ \scr{Q} \leftarrow \ch\gen{s^{-1}G}\to  \ch(\scr{P}).$$  We denote the chain operad pushforward of these maps by $\lu_{\varphi}(\scr P, \scr Q)$ or just $\lu(\scr P, \scr Q)$ when $\varphi$ is clear, and call it the \emph{universal linking operad}.  If $Q$ is not a differential operad, assume it is equipped with the trivial differential.  Provided $G\to\scr{P}$ is injective and $\varphi$ injects into $\ker\partial$, there are injections $\scr P \to \lu(\scr P, \scr Q)$ and $\scr Q\to \lu(\scr P, \scr Q)$.  We will use the notation $\lu_G(\scr{O}) = \lu(\gen{s\,G},\scr O)$ where $G$ is a generating $\Sigma$-module for the operad $\scr{O}$.  This is functorial in the category of operads paired with generating modules, where morphisms must take one generating module into another.  We simplify our notation to $\lu(\scr{O}) = \lu_\Gen(\scr{O})$ when there is an obvious choice of $\varphi$, such as when the generators are concentrated in a single arity.  However keep in mind that $\Gen$ is not functorial.  One should also note that the universal linking functor preserves coproducts because it is built out of a left adjoint and a colimit.  

The name universal linking operad has been chosen because any chain complex operad $\scr{O}$ which contains both $\scr{P}$ and $\scr{Q}$ and for which the generators of $\scr{P}$ differentiate to proper elements of $\scr{Q}$, (thus we think of $\scr{O}$ as linking $\scr{P}$ to $\scr{Q}$) there is a unique map (given by universal properties for ch and coker) $\lu(\scr{P},\scr{Q})\to \scr{O}$.  In other words, the universal linking operad is universal among operads which link $\scr{P}$ to $\scr{Q}$.  

\subsection{Homology Computations}
It is not hard to find examples of universal linking operads with nontrivial homology.  For instance, $\lu(\scr{A}ss)$ has nontrivial homology class given by $$\rho = s\mu (\mu\otimes 1 - 1\otimes \mu)$$ where $\mu$ is the generator for $\scr{A}ss$.  We will see that the homology of $\lu(\scr{A}ss)$ is free on this generator, and we will generalize the fact that this generator is the suspension of the associator.

In this section, we give explicit combinatorial descriptions of the homology of universal linking operads.  To begin, suppose $\scr{O}$ is an operad with generating $\Sigma$-module $G$.  Let $R$ denote the $\Sigma$-module suspension of $\ker\partial \cap\ker p$ where $p: \ch\gen{sG}\to \lu_G\scr{O}$ and $\partial$ is the chainification differential, and let $\partial R$ denote the desuspension of $R$.  Suppose we choose a chain map $\iota: R\oplus\partial R \to \ch\gen{sG}$ which restricts to be the kernel of $p$.  Furthermore, we consider $\partial R \oplus G$ a generating module for $\gen{G}$ by setting $\partial R$ to zero.

We will now define an operad which will turn out to be isomorphic to the homology of $\lu_G(\scr{O})$ Let $a$ be the inclusion of $\ker p$ into $\ch\gen{sG}$ and let $b$ be the inclusion of $\ker\partial \subseteq \ch\gen{sG}$ into $\ch\gen{sG}$.  Let $f$ be the map $b+c: \ker p \oplus\ker\partial \to \ch\gen{sG}$, and let $g$ be the map $\gen{R}\to \ch\gen{G}$ induced by $\iota$.  Finally, we define $\scr{R}_G(\scr{O}) = \text{coker}(g^*f)$.

\begin{prp}
For all operads $\scr{O}$ generated by $G$, there is an isomorphism \emph{$$H_*(\lu_G \scr{O})\to \scr{R}_G(\scr{O}).$$} 
\end{prp}

\begin{proof}
We have the following commutative diagram with exact rows and columns.  Here $(A)_\scr{P}$ would denote the ideal generated by $A$ in $\scr{P}$.
$$ \text{
\includegraphics[scale = 0.3]{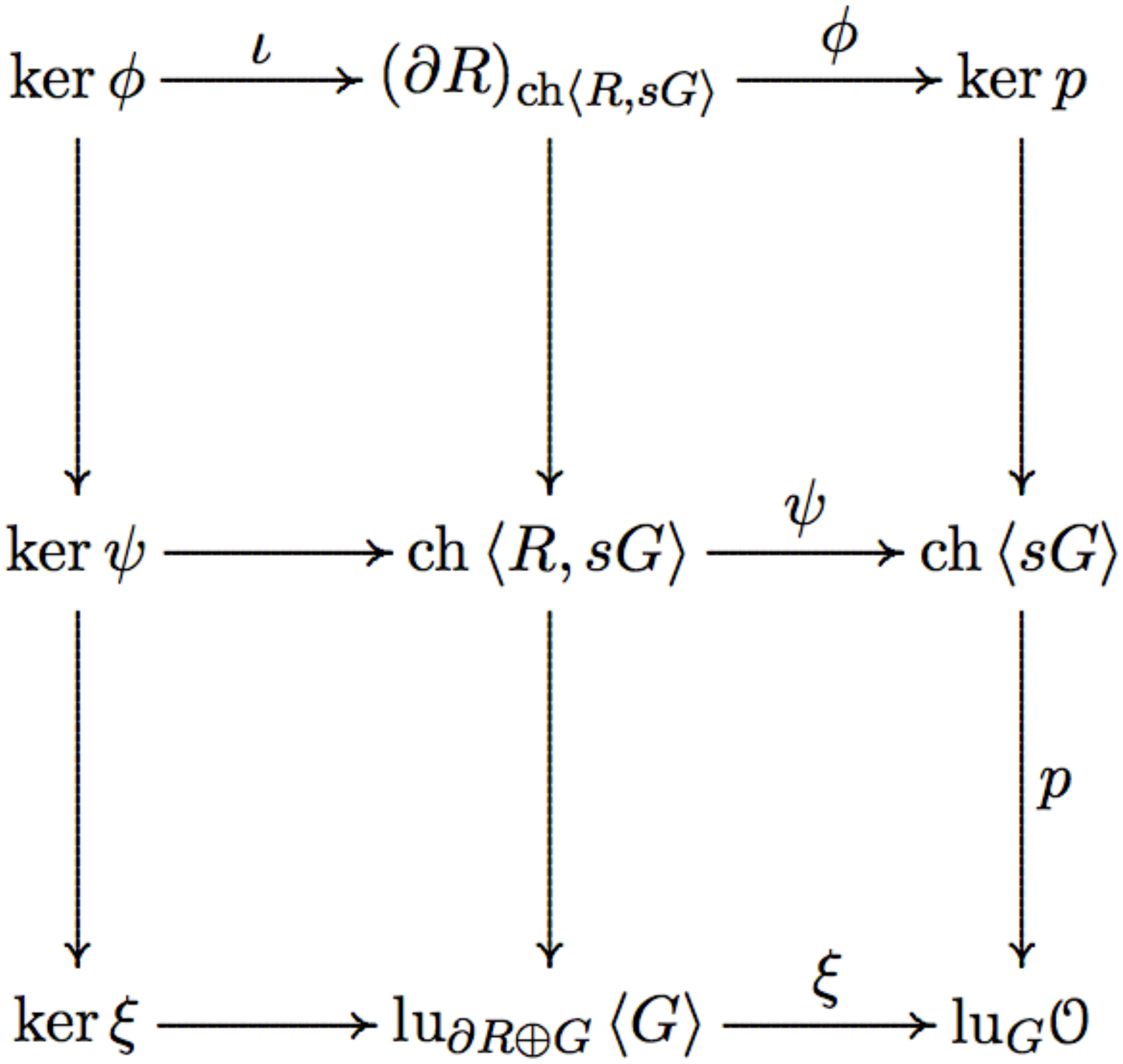}
} $$
In this diagram, $\phi$, $\psi$, and $\xi$ are induced by the inclusion of $R$ into $\ch\gen{sG}$.  By the long exact sequence in homology, we see that the homology of anything in the middle row is trivial, and therefore the homology of things in the top row are differentials of the homology of things in the bottom row.

We will now show the homology of $(\partial R)_{\ch{\gen{R,sG}}}$ consists of differentials of elements of $\gen{R}$.   Suppose that $k$ is a homology class for $(\partial R)_{\ch{\gen{R,sG}}}$.  Then for some $h\in \ch{\gen{R,sG}}$, we have $k = [\partial h]$.  Without loss of generality, $h$ is a linear combination of composition trees which are linearly independent modulo $\ker\partial$ and which are chosen so that each tree contains an element of $R$ as one of its nodes, and no tree contains an element of $\partial R$.  Looking at the definition for the differential for the chainification, we see that we can write $\partial = \partial_r + \partial_s$ where $\partial_r$ only differentiates the $R$ nodes and $s$ only differentiates the $sG$ nodes.  Also, as differentiating does not change tree shape, it is easy to see that the homology of $\ch{\gen{R,sG}}$ with respect to $\partial_s$ is given by the elements which don't contain any $sG$ nodes.  However, as $\partial h \in (\partial R)_{\ch{\gen{R,sG}}}$, we have that $\partial_s h = 0$ so there exists a $g$ such that $h - \partial_s g \in \gen{R}$.  Furthermore we have $\partial h = \partial_r h = \partial(h - \partial_s g) + \partial_r \partial_s g $ and we see that $\partial_s\partial_r\partial_s = 0$.  This allows us to see that $\partial_s$ is defined on $\ker \partial_r \cap (\partial R)_{\ch{\gen{R,sG}}}$ and since $\partial_s$ does not change tree shapes or alter nodes from $R$ or $\partial R$, the homology of $\ker \partial_r \cap (\partial R)_{\ch{\gen{R,sG}}}$ with respect to $\partial_s$ consists of the elements which have no nodes from $sG$.  We may therefore choose a $g'\in \ker \partial_r \cap (\partial R)_{\ch{\gen{R,sG}}}$ so that and $\partial_sg' = \partial_r\partial_s g$.  Finally we get that $\partial h = \partial(h-\partial_sg) + \partial g'$ which implies $k$ is the class of a differential of something in $\gen{R}$.  

We will now show that the homology of $\ker \phi$ consists of differentials of elements of the kernel of the map $\pi: \gen{R}\to \scr{R}$.  There is a basis for $\ker\psi$ given by applying $\text{id}-\psi$ to each basis tree in $\ch\gen{R,sG}$.  If we look at the span of these basis trees mapped into $\lu_{\partial R\oplus G} \gen{G}$, which must homologously give the suspension of $\ker\phi$, we get $$ \ker p \oplus_{\text{vect}} (\text{id}-\psi)(R)_{\lu_{R\oplus sG} \gen{G}} \hookrightarrow \lu_{\partial R\oplus G} \gen{G}$$ This is isomorphic as a graded vector space to $\ker p \oplus (R)_{\gen{G\oplus sG\oplus R}}$ but with the differential given by $\partial = \partial_g - \partial_r + \partial_s$, where $\partial_g$ differentiates the $\gen{G\oplus sG}$ part with the chainification differential, $\partial_r$ goes from the $\gen{G\oplus sG\oplus R}$ part to the $\gen{G\oplus sG}$ part and is given by differentiating only the nodes from $R$ with the chainification differential and then projecting to $\gen{G\oplus sG}$.  Lastly, $\partial_s$ is on the $\gen{G\oplus sG\oplus R}$ part and differentiates the nodes not from $R$ with the chainification differential.  There is an inclusion $u$ of $\ker p$ into this chain complex with chain cokernel $\lu_{\partial R\oplus G}\gen{G}$.  We see that any $h\in \ker p$ with $\partial h = 0$ will be in $\iota \partial R$ and have $\partial_gu h = 0$ and $h$ can be expressed as $\partial(\iota sh)$.  Then $ \iota\partial_r sh = u h$ and $\partial_s sh = 0$.  Finally this allows us to conclude from the homology long exact sequence that $\phi$ is injective in homology which means the nontrivial homology classes in $\ker\phi$ are the elements that map to nontrivial homology classes in $(\partial R)_{\ch{\gen{R,sG}}}$.  We reach the desired conclusion by identifying $\ker\pi$ with the $\gen{R}$-parts of the cyclic elements of $\ker p \oplus (R)_{\gen{G\oplus sG\oplus R}}$ with the differential we defined earlier.

Finally we have what we need for the homology long exact sequences of the diagram to give the isomorphism $H_*(\lu_G \scr{O})\to \scr{R}$.

\end{proof}

One will find that any element of $R$ that is not a composition of elements in the minimal generating $\Sigma$-module for the ideal will go to zero in $\scr{R}$.  This means that $\scr{R}$ is generated by the minimal relations for $\scr{O}$ with respect to the generating $\Sigma$-module $G$. It is fairly easy to explicitly carry out the construction of $\scr{R}$ for simple operads.  Doing so yields the following homologies of universal linking operads.

\begin{exm}
\emph{$H_*\lu\, \scr{A}ss$} is the free operad generated by a degree $1$ arity $3$ element, and that element is the class given by $s\mu(1\otimes\mu) - s\mu(\mu\otimes 1)$.
\end{exm}

\begin{exm}
\emph{$H_*\lu_G(\scr{C}omm) = \scr{L}ie$}  with arity equal to one more than the degree, where $G$ is the free $\Sigma$-module on one arity two generator.
\end{exm}

\begin{exm}
\emph{$H_*\lu(\scr{L}ie)$} is free on a single degree one, arity three generator which is invariant under cyclic permutations but not odd permutations.
\end{exm}

\begin{cnj}
\emph{$H_*\lu(\scr{C}omm,\scr{A}ss)$} is free on a degree one arity three generator for which the Jacobi identity holds.
\end{cnj}

We suspect this due to the following cokernel sequence. $$\lu(\scr{L}ie,0)\to \lu(\scr{A}ss) \to \lu(\scr{C}omm,\scr{A}ss)$$  

\subsection{Chain Complex Algebras}

\begin{dfn}
An operad $\scr{O}$ is integral if for all $a\in \scr{O}$ and all $b\in \scr{O}$ if $a\circ_1 b = 0$ then $a = 0$ or $b= 0$.
\end{dfn}
 
The reader should note that virtually all notable operads in $\text{Vect}_{\mathbb{K}}$ are both integral and quadratic, for instance the associative operad $\scr{A}ss$, the operad for Lie-algebras $\scr{L}ie$, pre-Lie algebras $\scr{P}re\scr{L}ie$, and the operad of commutative algebras $\scr{C}omm$.  

\begin{thm}
Let $\scr{O}$ be an integral operad which is concentrated in degree zero and generated by $G$.  Suppose that $C$ is a chain complex which has an algebra structure over $\scr{O}$ which acts trivially on $H_*C$.  Then $H_* C$ has a (potentially nontrivial) algebra structure over \emph{$H_* \lu_G(\scr{O})$}, which is unique if \emph{$H_1\text{End}(C)$} vanishes in arity equal to that of $G$.
\end{thm}

\begin{proof}
$C$ has an algebra structure over an operad concentrated in degree zero, $\mathscr{O}$, given by the structural morphism $$\alpha_n: \scr{O}(n) \to \hom_{\text{chain}}(C^{\otimes n},C)$$ and we have inclusions $\iota_n: \hom_{\text{chain}}(C^{\otimes n},C)\to \hom_{\text{vect}}(C^{\otimes n},C)$, the codomains of which form a chain complex operad denoted $\text{End}( C)$ with differential given by commutation with the differentials of its domain and codomain.  Since $\scr{O}$ acs trivially on $H_*(C)$, we know that the image of $\mathscr{O}$ in $\text{End}(C)$ is in the image of the differential, so, if $g$ denotes the inclusion of $\Gen(\scr{O})$ into $\scr{O}$, we can choose a lifting of $\iota\alpha g$ up $\partial$.  Therefore, we have a map $\lu \scr{O} \to \text{End}(C)$ given by the universal property, and since the cycles are the chain morphisms and the boundaries are the maps that are trivial on homology, we see that this induces maps $\xi_n: H_*\lu_G \scr{O} (n) \to \hom (H_*C^{\otimes n},H_*C).$  This is therefore a $H_*\lu \scr{O}$ operad algebra structure on $H_*C$.   The lifting of $\iota\alpha g$ will be unique up to homotopy if $H_1\text{End}( C)$ vanishes, and a change in lifting of $\iota\alpha g$, due to integrality, will induce the same maps $\xi_n$, and therefore the same algebra structure on $H^*(C)$.  This yields the uniqueness result.
\end{proof}

\begin{exm}
Consider a topological semigroup $A$ with $\tilde{H}_n(A;\mathbb{K}) \neq 0$ only if $n = 1 \mod 3$.  Then there is a trinary operation on the homology of the semigroup given by Theorem 3.1, which is nontrivial in some cases.  
\end{exm}

\begin{proof}
It is clear that any homology chain complex, in this situation, satisfies the conditions of Theorem 3.1 for the operad of associative algebras, and from Example 1, we see this yields a trinary operation on the homology.  It remains to construct a nontrivial example.  Consider the pointed CW complex given by $A = S^1\vee D^3\vee S^4$, and the pointed topological semigroup structure given by a homeomorphism $S^1\wedge S^1\to \partial D^3$ and homeomorphisms $S^1\wedge D^3\to S^4_+$ and $D^3\wedge S^1\to S^4_-$ where $S^4_+$ and $S^4_-$ are two hemispheres of $S^4$.  All other smash products are taken to the base point.  On the cellular homology chain complex, this induces an $\scr{A}ss$-algebra structure which acts trivially on the homology and therefore yields $\rho : \tilde{H}_1(A;\mathbb{K})^{\otimes 3}\to \tilde{H}_4(A;\mathbb{K})$, which we can compute as follows, using the notation that square brackets denote homology class and angle brackets denote the chain complex generator from the cell. \begin{align}
\rho([S^1]\otimes [S^1]\otimes [S^1]) = [s\mu(\mu(\gen{S^1}\otimes \gen{S^1})\otimes \gen{S^1} - \gen{S^1}\otimes\mu(\gen{S^1}\otimes \gen{S^1}) )]\nonumber\\
= [s\mu(\gen{\partial D^3}\otimes \gen{S^1} - \gen{S^1}\otimes \gen{\partial D^3})] = [\mu(\gen{D^3}\otimes \gen{S^1}) - \mu(\gen{S^1}\otimes \gen{D^3})]\nonumber\\
 = [\gen{S^4_+} - \gen{S^4_-}] = [S^4] \nonumber
\end{align}
Therefore $\rho$ is nontrivial, and we have constructed an explicit geometric example of when the construction from Theorem 3.1 is nontrivial.  

\end{proof}

\begin{exm}
Consider a topological space $A$ for which $\tilde{H}^n(A;\mathbb{K}) \neq 0$ only if $n = 1 \mod 3$.  Then there is a triple product $$\rho: \tilde{H}^k(A; \mathbb{K})\otimes \tilde{H}^m(A; \mathbb{K}) \otimes \tilde{H}^n(A; \mathbb{K})\to \tilde{H}^{k + m + n + 1}(A; \mathbb{K})$$ given by Theorem 3.1 applied to the cup product on the suspension of $A$.  
\end{exm}

\begin{proof}
Easy consequence of Theorem 3.1
\end{proof}

One should note that in the characteristic zero case, this operation is always trivial due to the properties of the cup product on the suspension.  However, we have yet to determine if there are nontrivial examples in the characteristic 2 case.  

If $C$ is a chain complex with $H^1(\text{End}(C)) = 0$ and an operad algebra structure for the integral operad $\scr{O}$ with structural map $\phi: \scr{O}\to \hom(C^{\otimes n},C)$.  Then this induces an $\scr{O}$ structure on $H_*(C)$ with structural map denoted by $H\phi: \scr{O} \to \text{End}(H_*C)$.  By Theorem 3.1, there is a $H_*(\lu(\ker(H\phi)))$ algebra structure on $H_*C$.  This gives a way of gaining additional information from any operad algebra structure on a siutable chain complex.

\section{Higher Linking Operads}

So far we have only considered the case of universal linking operads where $\varphi$ is an inclusion of a generating module. However, there are many interesting situations in which this is not the case.  For instance, we can construct a universal linking operad $\lu(\gen{sR},\gen{G})$ for any pair of free operads with a morphism $\gen{R}\to\gen{G}$.  For the standard presentation of the associative operad, $\scr{A}ss = \gen{\mu,\partial \rho }$, let us consider the homology of $\lu(\gen{\rho},\gen{\mu})$.  Basis elements of this operad are linear combinations and permutations of trees which have nodes of arity two or three, and with degree equal to the number of arity three nodes.  This can be embedded in the cellular complex of the associahedra, and we can think of this operad geometrically in this way.  We see that $H_*\lu(\gen{\rho},\gen{\mu})$ is the homology of the subset of the associahedra (the $A_\infty$ operad) which includes only the cubic cells.  In that case, the homology is generated by the pentagonal cells and points in the associahedra.  From this, we see that the homology of this universal linking operad is very nontrivial indeed, as it is not even finitely generated.  

From this geometric interpretation, we see that when we apply Theorem 3.1 to this universal linking operad, we get the first obstruction to a given nonassociative product extending to an $A_\infty$ structure wth a given choice of associator, because if any of the pentagons cannot be filled in, then there is no chance of filling in the entirity of the $A_\infty$ operad.  

It becomes clear from looking at this associahedron example that it may be natural to link on more operads to this universal linking operad in order to get the higher cells in our complex.  To do this, we would want to universally link the pentagon to $\lu(\gen{\rho},\gen{\mu})$ by taking $\lu(\gen{\rho'},\lu(\gen{\rho},\gen{\mu}))$, where $\rho'$ is the pentagon cell.  To repeat this process, it will be convenient to define a functor $\lk(\scr{O}_1,\scr{O}_2,...,\scr{O}_n)$ recursively as $$\lk(\scr{O}_1,\scr{O}_2,...,\scr{O}_{n+1}) = \lu(\scr{O}_{n+1},\lk(\scr{O}_1,\scr{O}_2,...,\scr{O}_n))$$ which assumes specified morphisms from the desuspension of each operad's generating module to the previous link operad.  We write $\lk(\scr{O}_1,\scr{O}_2,... )$ to mean the colimit of the finite links.  

\begin{exm} 
There is an isomorphism \emph{$$A_\infty \leftrightarrow \lk(\gen{\mu},\gen{\rho_1},\gen{\rho_2},\gen{\rho_3},...)$$} where each $\partial \rho_i$ is the boundary of the $i$-dimensional associahedron in terms of the other associahedra, and $\mu$ is the point.   
\end{exm}

\begin{proof}
The right hand side has basis given by arbitrary trees with permutation, which gives an isomorphism to $A_\infty$.  See \cite{vogt}\cite{muro}.  The differentials are equal by construction.  
\end{proof}

Let us discuss the process involved in this example in more depth.  Each time we add another operad to the link, we are killing all of the nontrivial homology classes of the universal linking operad that we already have.  This happens because the ideal generated by the new relation includes all of the nontrivial homology classes of the current link operad.  New homology classes arise in the cokernel, which are killed in the next link.  Each time we kill homology the nontrivial part of the homology vanishes entirely in one higher arity.  Therefore, after killing homology forever, we reach an operad which is acyclic except for the point classes that were never killed originally.  Therefore, the colimit is a chain complex operad which has homology equivalent to the original operad.  We will now generalize the example.

\begin{prp}
Let $\scr{O}$ be an operad and let $\gen{G|R}$ be its presentation.  Then define $\Sigma$-modules recursively as follows.  $R_0 = R$, and $R_{n+1} = \ker(f_n)(d_n)$, where \emph{$$f_n: H_*\lk(\gen{G},\gen{sR_0},\gen{sR_1},\gen{sR_2},...,\gen{sR_n}) \to \scr{O}$$} sets each $R_i$ to zero, and where $d_n$ is the smallest arity in which $\ker(f_n)(d_n)$ is nontrivial. The defining morphisms for the link operad can be any maps that take the homology classes to their representations.  Then \emph{$$  \lk(\gen{G},\gen{sR_0},\gen{sR_1},\gen{sR_2},...)$$}
is a quasi-free resolution of $\scr{O}$.
\end{prp}

\begin{proof}
We see this operad is quasi-free because it is a link of free operads.  To see that it is equvalent to $\scr{O}$, observe that the homology classes in the kernel of the map to $\scr{O}$ are trivialized by construction in one higher arity for each step in the colimit. The equivalence for the resolution is $\text{colim}_nf_n$.
\end{proof}

\begin{prp}
Any operad of the form \emph{$\lk(\gen{M_0},\gen{M_1},\gen{M_2},...)$} is a minimal operad if for all $i < j$, we have \emph{$\text{max}\,\text{arity}(M_i) < \text{min}\,\text{arity}(M_j)$}.  
\end{prp}

\begin{proof}
The decomposability of the differential follows from the increasing arity condition, and the diferential of $M_{n+1}$ is in $\gen{\oplus_{i = 0}^n M_i}$ by definition.  
\end{proof}

\begin{thm}
If $\gen{G|R}$ is a minimal presentation for $\scr{O}$, then the resolution given by Proposition 4.1 is the minimal model.  
\end{thm}

\begin{proof}
From the proof of Proposition 4.2, we see that the conditions for minimality hold for $R_i$ with $i > 1$.  The only possible issue is that the differentials of $sR_1$ and $sR_0$ might not be decomposable for an arbitrary presentation.  We see that if $G$ is the minimal generating module, then the differential of $sR_0$ is decomposable.  Any homology class built from the elements of $R_0$ and $G$ with a nontrivial tree-basis component $\rho\in R$ will yield, via the basis components that cancel with $\partial \rho$ in the differential, an expression for the relation $\rho$ in the ideal generated by the other relations in $R$.  As we assume that the relations we chose are minimal, such a homology class can not exist.  Therefore, the derivative of $sR_1$ is decomposable as well.  This gives minimality.
\end{proof}

Minimal models are defined in \cite{markl} and \cite{loday} and are unique up to isomorphism.

\begin{prp}
All quasi-free chain complex operads $\scr{F}$ are equivalent to some link of free operads.
\end{prp}

\begin{proof}
Take $\Gen(\scr{F})$ and split it into $G_0,G_1,...$ by degree.  Then the operad given by $ \lk(\gen{G_1},\gen{G_2},...) $ will be isomorphic to $\scr{F}$ as vector space operads because they are both free on the same generators.  The chain maps are the same by construction because we can choose the differentials to be anything that maps to lower degree.
\end{proof}

\newpage

\bibliographystyle{plain}

\bibliography{refs}

\end{document}